\documentclass[a4paper,final]{siamltex}
\usepackage[latin1]{inputenc}
\usepackage{amsfonts,amssymb,amsmath,amssymb}
\usepackage[notref,notcite]{showkeys}

\usepackage{graphicx,epsfig,verbatim,url}

\usepackage{exscale,path}

\newtheorem{Prop}[theorem]{Proposition}

\newtheorem{thm}[theorem]{Theorem}

\newtheorem{Cor}[theorem]{Corollary}

\newtheorem{Rem}[theorem]{Remark}
\newtheorem{Test}[theorem]{Test}

\newcommand{\R}{\mathbb{R}}
\newcommand{\C}{\mathbb{C}}

\newcommand{\wt}{\widetilde}
\newcommand{\wh}{\widehat}

\newcommand{\iunit}{{\sf {\bf i}}}

\newcommand{\la}{\lambda}

\newcommand{\g}{\gamma}

\newcommand{\ds}{\displaystyle}

\newcommand{\twoone}[2]{\left[\begin{array}{c} #1\\ #2 \end{array}\right]}
\newcommand{\twotwo}[4]{\left[\begin{array}{cc} #1 & #2 \\ #3 & #4 \end{array}\right]}

\renewcommand{\le}{\leqslant}
\renewcommand{\ge}{\geqslant}

\renewcommand{\theta}{\vartheta}
\newcommand{\wvp}{\wt \varphi}

\DeclareMathOperator{\opvec}{vec}
\DeclareMathOperator{\polar}{polar}

\begin{document}

\title{The geometric mean of two matrices %
	\\from a computational viewpoint %
}
\author{Bruno Iannazzo\thanks{ %
Dipartimento di Matematica e Informatica, Via Vanvitelli 1, Perugia, Italy. 
} %
}

\pagestyle{myheadings}\markboth{\sc Bruno Iannazzo 
} 
{\sc Geometric Mean of two Matrices
}

\maketitle
\begin{abstract}
The geometric mean of two matrices is considered and analyzed from a computational viewpoint. Some useful theoretical properties are derived and an analysis of the conditioning  is performed. Several numerical algorithms based on different properties and representation of the geometric mean are discussed and analyzed and it is shown that most of them can be classified in terms of the rational approximations of the inverse square root functions. A review of the relevant applications is given.
\end{abstract}


\section{Introduction}\label{sec:intro}

The geometric mean of two positive numbers $a$ and $b$ is defined as $\sqrt{ab}$. The adjective ``geometric'' is referred to the fact that the geometric mean is the length of the edge of a square  having the same area as a rectangle whose edges have length $a$ and $b$, respectively.

A typical wish in mathematics is to generalize concepts as much as possible. It is then understood why researchers have tried to generalize the concept of geometric mean to the matrix
generalizations of positive numbers, namely Hermitian positive definite matrices. 
We denote by $\mathcal P_n$ the set of $n\times n$ Hermitian positive definite matrices, which we will call just positive matrices. The geometric mean of two matrices need to be a function $\varphi:\mathcal P_n\times \mathcal P_n\to \mathcal P_n$.

The generalization is not trivial, since the formula $\sqrt{ab}$, applied to matrices would lead to the definition $\psi(A,B):=(AB)^{1/2}$, which is unsatisfatory since, for instance, $\psi(A,B)\ne \psi(B,A)$. A different, more fruitful, approach to get a fair generalization is axiomatic, that is derive the definition of geometric mean from the properties it ought to satisfy.

A natural property required by a generalization is the following: given a diagonal matrix $D=\diag(d_1,\ldots,d_n)$, with $d_i>0$, and the identity matrix $I$, the geometric mean is $\varphi(D,I):=\diag(\sqrt{d_1},\ldots,\sqrt{d_n})$. The aforementioned property is referred as {\em consistency with scalars}. 
  
The consistency with scalars is not sufficient to uniquely define a geometric mean. We need another property, namely the {\em congruence invariance}: let $A,B\in\mathcal P_n$ and $S$ belonging to the set $GL(n)$ of invertible matrices of size $n$, then 
$	\varphi(S^*AS,S^*BS)=S^*\varphi(A,B)S.   $
The congruence invariance is mathematically relevant since it states that the geometric mean interplay well with the action of $GL(n)$ over $\mathcal P_n$, that is the congruence. Moreover, it allows the geometric mean to model physical quantities.

The following is a minor variation of a result of Bhatia \cite[Sec. 4.1]{bhatia}.

\begin{theorem} Let $\varphi:\mathcal P_n\times \mathcal P_n\to \mathcal P_n$ be a function which verifies both consistency with scalars and congruence invariance, then 
\begin{equation}\label{eq:geomeandef}
	\varphi(A,B)=A^{1/2}(A^{-1/2}BA^{-1/2})^{1/2}A^{1/2}=:A\# B.
\end{equation}
\end{theorem}

The symbol $A^{1/2}$ stands for the principal square root of the
matrix $A$, which is a matrix satisfying the equation $X^2=A$ and
whose eigenvalues have positive real part. Such a matrix exists
and is unique if $A$ has no nonpositive real eigenvalues, in particular if $A$ is positive then $A^{1/2}$ is positive. Moreover, for any invertible matrix $M$, it holds that $M^{-1}A^{1/2}M=(M^{-1}AM)^{1/2}$ (see \cite{highambook}).

It can be proved that $A\# B$ verifies all the other properties required by a geometric mean, for instance $A\# B=B\# A$, and if $A$ and $B$ commute, then $A\# B=(AB)^{1/2}$. Thus, the definition is well established.

Notice that $A\#B$ solves the Riccati equation $XA^{-1}X=B$ and it can be proved that it is the unique positive solution \cite[Thm. 4.1.3]{bhatia}. Moreover, using the properties of the principal square root one can derive
\begin{equation}\label{eq:mean2equiv}
A\#B=A(A^{-1}B)^{1/2}=(BA^{-1})^{1/2}A
=B(B^{-1}A)^{1/2}=(AB^{-1})^{1/2}B.
\end{equation}

Yet another important property of the geometric mean can be given in terms of a special Riemannian geometry of $\mathcal P_n$. The geometry is obtained by the scalar product $\langle X,Y \rangle_A=\mathrm{trace}(A^{-1}XA^{-1}Y)$ on the tangent space $T_A\mathcal P_n$ at a positive matrix $A$ (which is the set of Hermitian matrices). 
In the resulting Riemannian manifold there exists only one
geodesic, $\gamma:[0,1]\to \mathcal P_n$, joining any two positive definite matrices $A$ and $B$ and whose explicit expression is known to be \cite[Thm. 6.1.6]{bhatia}
\begin{equation}\label{eq:geodesic}
A\#_t B:=\gamma(t)=A(A^{-1}B)^t=A^{1/2}(A^{-1/2}BA^{-1/2})^tA^{1/2}.
\end{equation}
It is now apparent that $A\# B=A\#_{1/2} B$ is the mid-point of the geodesic joining $A$ and $B$.

The definition $A\# B=A(B^{-1}A)^{-1/2}$ in terms of an inverse square root
yields a rather large number of integral representations \cite{im:palin} among
which we note the following \cite{alm}:
\begin{equation}\label{eq:gm-int}
  A\#B=\frac 1{\pi}\int_0^1
  \frac{(tB^{-1}+(1-t)A^{-1})^{-1}}{\sqrt{t(1-t)}}\,dt.
\end{equation}

The relevant applications of the geometric mean of two matrices are reviewed in Section \ref{sec:appl}.

The contributions of the paper are of different kind.  First of all, we investigate some simple theoretical properties of the geometric mean of two matrices, giving a new formula for $A\# B$ in terms of the polar decomposition and an expression of $A\# B$ in terms of polynomials in $A^{-1}B$ and $B^{-1}A$ which are useful for computational purposes. Then, we discuss the sensitivity (in the Euclidean sense) of the matrix geometric mean function with respect to perturbations getting upper and lower bounds for the condition number. Then, we devote a large part to the old and new algorithms for the geometric mean and related quantities like $A\#_t B$.

The existing methods are the
averaging technique of Anderson and Trapp \cite{at}, a method based on the matrix sign function of Higham et al. \cite{hmmt05}, the palindromic cyclic reduction of Iannazzo and Meini \cite{im:pcr} and a method based on a continued fraction expansion of Ra\"issouli and Leazizi \cite{rl}. We  show that the sign method and the palindromic cyclic reduction are two variants of the averaging  technique.

We present some further algorithms for the matrix geometric mean: the first one is based on the Cholesky factorization and the Schur decomposition and performs with great numerical stability in practice; the second is based on the expression of $A\# B$ in terms of the polar decomposition of certain matrices and is attractive since it relies on the small computational cost of the polar factor in terms of arithmetic operations (ops);
the third is a Gaussian quadrature applied to the integral
representation \eqref{eq:gm-int}; while the fourth is based on the rational minimax approximation to the inverse square root
which is essentially the algorithm of Higham, Hale and Trefethen \cite{triplenick}.

A perhaps surprising property is that the polar decomposition algorithm and the Gaussian quadrature, in their basic definition, produce the same sequence as the averaging
technique and so they can be seen as yet two more variants of it.
Moreover, they can be described in terms of certain Pad\'e approximation at $x_0=1$ of the inverse square root.

The organization of the paper is as follows. In the next section we give a couple of properties of the geometric mean which will be useful later. In Section \ref{sec:cond} we compute the condition number of the matrix geometric mean. In Section \ref{sec:schur} we discuss the Cholesky-Schur algorithm. In Section \ref{sec:pade} we discuss the algorithms related to the Pad\'e approximation of $z^{-1/2}$ while in Section \ref{sec:minimax} we discuss the ones related to its rational minimax approximation. In Section \ref{sec:appl} we review the applications where a matrix geometric mean is required. In Section \ref{sec:test} we perform some numerical tests, while in Section \ref{sec:conclusions} we draw the conclusions.

Now, we recall some concept and facts that will be used in the paper. 
We recall that any nonsingular matrix $M$ can be written as $HU$ where $H$ is Hermitian and $U$ is unitary; the latter is called the polar factor of $M$, denoted by polar$(M)$, and whose explicit expression is $U=M(M^*M)^{-1/2}$. Given two matrices $M$ and $N$ we denote by $M\otimes N$ their Kronecker (tensor) product and by $\opvec(M)$ the vector obtained stacking the columns of $M$. We speak of vec basis for $\C^{n\times n}$ as the basis in which the coordinates of a matrix $M$ are $\opvec(M)$, similarly the vec basis for $\C^{n\times n}\times \C^{n\times n}$ is the one in which the coordinates of $(M,N)$ are $\twoone{\opvec(M)}{\opvec(N)}$. Finally, let $f(A)$ be a matrix function, then for any invertible matrix $M$, it holds that
\begin{equation}\label{eq:siminv}
f(MAM^{-1})=Mf(A)M^{-1};
\end{equation}
we call this property {\em similarity invariance} of matrix functions. Beside similarity invariance, we use several other properties of general and specific matrix functions, for this topic we address the reader to the book of Higham \cite{highambook}.


\section{Some properties of the geometric mean}

Any positive matrix $A$ can be written as $A=C^*C$ for an invertible $C$. Two noticeable examples are $A=A^{1/2}A^{1/2}$ and the Cholesky factorization $A=R^*R$, where $R$ is upper triangular with positive diagonal entries.

Given two positive matrices $A$ and $B$, with factorizations $A=C^*C$ and $B=D^*D$, the matrix geometric mean of $A$ and $B$ can be characterized using the following result which generalizes Proposition 4.1.8 of \cite{bhatia}.

\begin{Prop}\label{thm:fact}
Let $A=C^*C$ and $B=D^*D$ with $C,D\in\C^{n\times n}$ nonsingular. Then 
\begin{equation}\label{eq:fact}
	A\# B=C^*\polar(CD^{-1})D,
\end{equation}
where $\polar(CD^{-1})$ is the unitary polar factor of $CD^{-1}$.
Moreover, let $U$ be a unitary matrix such that $C^*UD>0$, then $C^*UD=A\# B$ and $U=\polar(CD^{-1})$.
\end{Prop}
\begin{proof}
Using the formula $\polar(M)=M(M^{-1}M^{-*})^{1/2}$ and the similarity invariance of the square root \eqref{eq:siminv} we get
\begin{multline}\nonumber
	C^*\polar(CD^{-1})D=C^*CD^{-1}(DC^{-1}C^{-*}D^*)^{1/2}D\\=
	A(D^{-1}DA^{-1}D^*D)^{1/2}=A(A^{-1}B)^{1/2}=A\# B.
\end{multline}
The second statement can be obtained suitably modifying the proof of Proposition 4.1.8 of \cite{bhatia}.
\end{proof}

Yet another interesting property is obtained using the fact that the principal square root of a matrix $Z\in\C^{n\times n}$ is a polynomial in $Z$ \cite{highambook}. In particular, if $Z$ has real positive eigenvalues, then $Z^{1/2}=p(Z)$ where $p(z)$ is the polynomial interpolating the points $(\la_1,\sqrt{\la_1}),\ldots,(\la_s,\sqrt{\la_s})$, where $\la_1,\ldots,\la_s$ are the distinct eigenvalues of $Z$.

Since $A\# B=A(A^{-1}B)^{1/2}=B(B^{-1}A)^{1/2}$, we get the following result.
\begin{Prop}\label{thm:poly}
Let $A,B \in\mathcal P_n$, and let $\la_1,\ldots,\la_s$ be the distinct eigenvalues of $A^{-1}B$, then $A\# B=Ap(A^{-1}B)=Bq(B^{-1}A)$, where $p(z)$ and $q(z)$ are the interpolating polynomials of the  points $(\la_i,\sqrt{\la_1}),\ldots,(\la_s,\sqrt{\la_s})$ and $(1/\la_i,1/\sqrt{\la_1}),\ldots,(1/\la_s,1/\sqrt{\la_s})$, respectively.
\end{Prop}

Proposition \ref{thm:poly} has some interesting consequences. First of all we get that if $A$ and $B$ are $2\times 2$ matrices then $A\#B=a_0A+a_1B$. An explicit expression of $a_0$ and $a_1$ is well known, in fact \cite[Prop. 4.1.12]{bhatia} 
\[
	A\# B=\frac{ \sqrt{\alpha\beta}}{\sqrt{\det(\alpha^{-1}A+\beta^{-1}B)}}
	(\alpha^{-1}A+\beta^{-1}B),\quad \alpha=\sqrt{\det(A)},\quad
	\beta=\sqrt{\det(B)}.
\]
Similarly, if $A$ and $B$ are such that $A^{-1}B$ has just two eigenvalues then $A\# B=a_0A+a_1B$. 

An application of Proposition \ref{thm:poly} concerns the preservation of matrix structures by the geometric mean. For instance, if $\mathcal G$ is an algebra of matrices, then $A,B\in\mathcal G\cap \mathcal P_n$ implies that $A\# B\in \mathcal G\cap \mathcal P_n$. An example of $\mathcal G\cap \mathcal P_n$ is the set of circulant Hermitian positive definite matrices.

Proposition \ref{thm:poly} holds also for $A\#_t B$ since $(A^{-1}B)^t$ as well is a polynomial in $A^{-1}B$. 
This fact allows us to prove that the Karcher mean of positive definite matrices (see \cite{bi} for the definition) preserves $\mathcal G\cap \mathcal P_n$, where $\mathcal G$ is an algebra of matrices.

\begin{Prop}\label{thm:subgroups}
Let $\mathcal G$ be an algebra of $n\times n$ matrices and let $A_0,\ldots,A_m\in \mathcal G\cap \mathcal P_n$, then the Karcher mean $G$ of $A_0,\ldots,A_m$ belongs to $\mathcal G\cap \mathcal P_n$.
\end{Prop}
\begin{proof}
The sequence $S_k=S_{k-1}\#_{1/k} A_{(k\mathrm{\ mod\ }m)+1}$, with $S_1=A_1$, converges to the Karcher mean \cite{holbrook}. From Proposition \ref{thm:poly} applied to $A\#_t B$ we get that $S_k$ belongs to $\mathcal G\cap \mathcal P_n$ for each $k$, and since $\mathcal G$ is closed we have that $G\in\mathcal G$. On the other hand, by the definition of Karcher mean $G\in\mathcal P_n$, and thus $G\in\mathcal G\cap \mathcal P_n$.
\end{proof}


\section{Conditioning}\label{sec:cond}

We describe the sensitivity of the matrix geometric mean function
$\varphi:\mathcal P_n\times \mathcal P_n \to \mathcal P_n\,:\, (A,B)\to A\# B$ to 
perturbations in both its arguments, $A$ and $B$. 
For any couple of positive matrices $(A,B)$ there exists a neighborhood $\mathcal U\subseteq \C^{n\times n}\times \C^{n\times n}$ of it in which the function $\varphi$ can be extended to a differentiable function $\wvp$ with the same formula $\wvp(X,Y)= X(X^{-1}Y)^{1/2}$, for $X,Y \in\mathcal U$. The
differential (Fr\'echet derivative) at a point $(A,B)$ is a linear function
$d\wvp_{(A,B)}:\C^{n\times n}\times \C^{n\times n}\to \C^{n\times n}$.

A measure of the sensitivity is given by the relative condition number
whose expression, following Rice \cite[Thm. 4]{rice}, 
is
\[
\mathrm{cond}(\wvp,(A,B))=\frac{\|(A,B)\|
\|d\wvp_{(A,B)}\|}{\|\wvp(A,B)\|}
\]
where the norm of the couple $(A,B)$ is the norm of the matrix $[A\ B]$ and the norm of the operator $d\wvp_{(A,B)}$ is defined in the
usual sense by
\[
\|d\wvp_{(A,B)}\|=\max_{H,K\mathrm{\ not\ both\ zero}}
\frac{\|d\wvp_{(A,B)}[(H,K)]\|}{\|(H,K)\|}.
\]

To give an explicit expression of the condition number from which
deduce suitable bounds, we need to
compute the differential of $\wvp$ at a couple $(A,B)$ . It may be useful the following expression of the
differential of the matrix mean function (extended in a neighborhood of $(A,B)$).
\begin{theorem}\label{thm:dermgm}
Let $A,B\in\mathcal P_n$ and $H,K\in \C^{n\times n}$, and let $\wvp$ be the extension of the matrix mean function in a neighborhood of $(A,B)$ in $\C^{n\times n}\times \C^{n\times n}$, then the following representation of
$D=d\wvp_{(A,B)}[(H,K)]\in \C^{n\times n}$ holds
\[
\opvec(D)=(I\otimes Z^{-1}+Z^{-1}\otimes I)^{-1}\opvec(H)
+(I\otimes Z+Z\otimes I)^{-1}\opvec(K),
\]
where $Z=(BA^{-1})^{1/2}$.
\end{theorem}
\begin{proof}
It is enough to find the ``partial'' derivative of $\wvp$ with
respect to a perturbation on $B$, say $K$, then interchanging $A$ and $B$
yields the full result.

Let $f(z)=z^{1/2}$,  recall that for any matrix with real positive eigenvalues $S\in\C^{n\times n}$ and any matrix direction $F\in\C^{n \times n}$, it holds that 
$\opvec(df_S[F])=(I\otimes S^{1/2}+(S^{1/2})^T\otimes I)^{-1}\opvec(F)$ \cite[Chap. 6]{highambook}.

Since, by the similarity invariance of the square root \eqref{eq:siminv}, $\wvp(X,Y)=X(X^{-1}Y)^{1/2}=
X(X^{-1}YX^{-1}X)^{1/2}=(YX^{-1})^{1/2}X$, for any $(X,Y)$ in a neighborhood of $(A,B)$, we have by the chain rule
\[
     d\wvp_{(A,B)}[0,K] =df_{BA^{-1}}[KA^{-1}]A
\]
which in the $\opvec$ basis can be written as
\[
\begin{split}
	\opvec(d\wvp_{(A,B)}[0,K])
& =(A\otimes I)(I\otimes Z+Z^T \otimes I)^{-1}\opvec(KA^{-1})\\
 &=(A\otimes I)(I\otimes Z+Z^T\otimes I)^{-1}(A^{-1}\otimes I)\opvec(K)\\
& =(I\otimes Z+AZ^TA^{-1}\otimes I)^{-1}\opvec(K)\\
& =(I\otimes Z+Z\otimes I)^{-1}\opvec(K),
\end{split}
\]
where we have used the fact that $AZ^TA^{-1}=A(A^{-1}B)^{1/2}A^{-1}=(BA^{-1})^{1/2}AA^{-1}=Z$. 
\end{proof}

Using Theorem \ref{thm:dermgm} and setting $M_1=(I\otimes
Z^{-1}+Z^{-1}\otimes I)^{-1}$ and $M_2=(I\otimes Z+Z\otimes I)^{-1}$ it is
possible to get the expression for the (relative) condition number in the
Euclidean (Frobenius) norm
\[
\mathrm{cond}(\wvp,(A,B))=\frac{
\|[M_1\ M_2]\|_2\|[A\ B]\|_F }
{\|
A\#B\|_F}.
\]
We have used the fact that the operator norm induced by the Euclidean norm coincides with the matrix 2-norm (spectral norm) of the matrix representation of the operator in the vec basis since
\[
\sup\frac{\|d\wvp_{(A,B)}[H,K]\|_F}{\|[H\ K]\|_F}
=\sup\frac{\|[M_1\ M_2]\opvec([H\ K])\|_2}{\|\opvec([H\ K])\|_2}=\|[M_1\ M_2]\|_2.
\]
The absolute condition number is $\kappa(A,B):=\|[M_1\ M_2]\|_2$.

{}From the  properties of the spectral norm, the following inequalities hold
\begin{equation}\label{eq:condbound}
\max\{\|M_1\|_2,\|M_2\|_2\}\le \|[M_1\ M_2]\|_2
\le (\|M_1\|_2^2+\|M_2\|_2^2)^{1/2}.
\end{equation}
To get bounds for the condition number, observe that there exists $K$
such that $D=K^{-1}ZK$ is diagonal and $K\otimes K$ diagonalizes both
$M_1$ and $M_2$. Thus, using \eqref{eq:condbound}, we get the bounds for the condition number in the Euclidean norm, which we denote by $\kappa_F(A,B)$,
\begin{equation}\label{eq:bounds}
\frac 12 \max\bigl\{\rho(Z),\rho(Z^{-1})\bigr\}
\le
\kappa_F(A,B)
\le
\frac 12\,\mu_2(K\otimes K)\bigl(\rho(Z)^2+\rho(Z^{-1})^2\bigr)^{1/2}.
\end{equation}

Let $U$ be a unitary matrix which diagonalizes $M=B^{-1/2}AB^{-1/2}$, that is $U^*MU$ is a diagonal matrix, then the matrix $V=B^{-1/2}U$ diagonalizes $B^{-1}A$ and thus diagonalizes $Z$ and $Z^{-1}$. Moreover, $\mu_2(V)=\mu_2(B^{-1/2})=\mu_2(B)^{1/2}$. We get an upper bound for $\mu_2(K\otimes K)$ as $\mu_2(V\otimes V)=\mu_2(B)$ and interchanging $A$ and $B$ we get a less sharp but better understandable upper bound for the condition number
\begin{equation}\label{eq:bounds2}
\kappa_F(A,B)
\le
\frac 12\,\min\{\mu_2(A),\mu_2(B)\}\sqrt{\rho(B^{-1}A)+\rho(A^{-1}B)}.
\end{equation}

Given $A$ and $B$ we can possibly reduce the bounds in
\eqref{eq:bounds} by a simple scaling of the matrices $A$ and $B$ by positive parameters $\alpha$ and $\beta$, getting the new matrices $\wt A=\alpha A$, $\wt B=\beta B$
and $\wt Z=\sqrt{\beta/\alpha}\,Z$. From $\wt A\#\wt B$ we obtain the required
geometric mean through $A\# B=\frac 1{\sqrt{\alpha\beta}}((\alpha A)\#(\beta B))$.

The choices of $\alpha$ and $\beta$ which minimize both $\sqrt{ \rho(\wt Z)^2+\rho(\wt Z^{-1})^2}$ and $\max\{\rho(\wt Z),\rho(\wt Z^{-1})\}$
are such that $\alpha/\beta=\rho(Z)/\rho(Z^{-1})=mM$ where $m$ and $M$
are the extreme eigenvalues of $Z$.  An approximate value of $\alpha/\beta$ can  be obtained by the approximations of $M^2$ and $m^2$ got by some
steps of the power and inverse power methods applied to $B^{-1}A$ (or $A^{-1}B$).


\section{An algorithms based on the Schur decomposition}\label{sec:schur}

We explain how to efficiently compute a point of the geodesic $A\#_t B$ using the Schur decomposition and the Cholesky factorization. The resulting algorithm can be used to compute the matrix geometric mean for $t=1/2$.

Consider the Cholesky factorizations $A=R_A^* R_A$ and $B=R_B^*R_B$. Using the similarity invariance of the matrix functions we get
\begin{equation}\label{eq:sharpchol}
	A\#_t B=A(A^{-1}B)^t=R_A^*R_A(R_A^{-1}R_A^{-*}BR_A^{-1}R_A)^t
	=R_A^*(R_A^{-*}BR_A^{-1})^tR_A,
\end{equation}
and thus, the evaluation of $A\#_t B$ can be obtained by forming the Cholesky decomposition of $A$, inverting the Cholesky factor $R_A$ (whose condition number is the square root of the one of $A$) and computing the $t$-th power of the positive definite matrix $V=R_A^{-*}BR_A^{-1}$. This is done by computing the Schur form $V=UDU^*$ and getting
\begin{equation}\label{eq:schur2}
	A\#_t B=R_A^*UD^tU^*R_A,\qquad R_A^{-*}BR_A^{-1}=UDU^*,
\end{equation}
The power of $D$ is computed elementwise.

If the condition number of $A$ is greater than the one of $B$, it may be convenient to interchange $A$ and $B$ in order to get a possibly more accurate results. Using the simple equality $A\#_tB=B\#_{1-t}A$, the formula is
\begin{equation}\label{eq:schur3}
	A\#_t B=B\#_{1-t} A=R_B^*UD^{1-t}U^*R_B,\qquad R_B^{-*}AR_B^{-1}=UDU^*.
\end{equation}

We synthesize the procedure.

\bigskip

{\noindent \bf Algorithm 4.1 (Cholesky-Schur method)} Given $A$ and $B$ positive definite matrices, $t\in(0,1)$, compute $A\#_t B$. 
\begin{enumerate}
\item if the condition number of $A$ is greater than the condition number of $B$ interchange $A$ and $B$, computing $B\#_{1-t}A$;
\item compute the Cholesky factorizations $A=R_A^*R_A$, $B=R_B^*R_B$ and form $V=R_A^{-*}BR_A^{-1}=X^*X$ where $X$ is the upper triangular matrix solving $XR_B=R_A$; 
\item compute the Schur decomposition $UDU^*=V$;
\item compute $A\#_t B=R_B^*UD^t U^*R_B$.
\end{enumerate}

\bigskip

The computational cost of the procedure is given by the Cholesky factorizations ($\frac 23n^3$ arithmetic operations (ops)), the computation of $V$  ($n^3$ ops), the Schur decomposition (about $9n^3$ ops), the computation of $R_B^*UD^t U^*R_B$ ($3n^3$ ops), for a total cost of about $(14+\frac 23)n^3$ ops.

All the steps of Algorithm 4.1 can be performed in a stable way, thus the resulting algorithm is numerically stable.

\begin{Rem}\rm
An alternative to compute $A\#_t B$ is to use directly one of the formulae 
\begin{equation}\label{eq:schur}
\begin{split}
A\#_t B & =A(A^{-1}B)^t=B(B^{-1}A)^{1-t}\\
&=A\exp(t\log(A^{-1}B))=A\exp(-t\log(B^{-1}A)).
\end{split}
\end{equation}
The expressions in the first row of \eqref{eq:schur} can be evaluated either by forming the Schur decomposition of the matrix $A^{-1}B$ (or $B^{-1}A$) which is nonnormal in the generic case or using the approximation algorithm of Higham and Lin \cite{higlin11}. Alternatively one could use the expressions in the second row of \eqref{eq:schur} where the exponential and the logarithm can be computed as explained in \cite{highambook}.
Unfortunately, none of these alternatives is of interest since they are more expensive than the Cholesky-Schur algorithm and do not exploit the positive definite structure of $A$, $B$ and $A\#_t B$.
\end{Rem}


\section{Algorithms based on the Pad\'e approximation of $z^{-1/2}$}\label{sec:pade}

We give three methods (with variants) for computing the matrix geometric mean, based on matrix iterations or a quadrature formula, two of them are apparently new.
The algorithms are derived using different properties of the
matrix geometric mean, however, perhaps surprisingly, they give
essentially the same sequences which can be also derived using certain Pad\'e approximation of $z^{-1/2}$ in the formula $A(B^{-1}A)^{-1/2}$.

The first method is based on the simple property
that iterating two means one obtains a new mean: the geometric
mean is obtained as the limit of an iterated arithmetic-harmonic
mean. The second is based on the polar decomposition and if the
robustness is the main concern it is possible to compute it in 
a backward stable way \cite{kz,hn}. The latter is based on an
integral representation of the matrix geometric mean computed with
a Gauss-Chebyshev quadrature, that method could be useful if one is
interested in the computation of $(A\#B) v$, for $A$ and $B$ large
and sparse.

\subsection{Scaled averaging iteration}\label{sec:avera}

Let $a$ and $b$ be two positive integers, their geometric mean
$\sqrt{ab}$ can be obtained as the limit of the sequences
$a_{k+1}=(a_k+b_k)/2$, $b_{k+1}=2a_kb_k/(a_k+b_k)$ with $a_0=a$
and $b_0=b$. The updated values $a_{k+1}$ and $b_{k+1}$ are the
arithmetic and the harmonic mean, respectively, of $a_k$ and $b_k$.

This ``averaging technique'' can be applied also to matrices
leading to the first, as far as we know, algorithm for computing $A\# B$ provided by
Anderson, Morley and Trapp \cite{amt} and based on the coupled
iterations
\begin{equation}\label{eq:2nd1}
    \left\{
    \begin{array}{l}
    A_0=A,\quad B_0=B,\\
    A_{k+1}=(A_k+B_k)/2,\\
    B_{k+1}=2A_k(A_k+B_k)^{-1}B_k=2(A_k^{-1}+B_k^{-1})^{-1},
    \end{array}
    \right. \quad k=0,1,2,\ldots,
\end{equation}
where $A_k$ and $B_k$, for $k=1,2,\ldots$, both converge to $A\# B$.
Observe that $A_{k+1}$ is the arithmetic mean of $A_k$ and $B_k$,
while $B_{k+1}$ is the harmonic mean of $A_k$ and $B_k$.

The convergence is monotonic in fact it can be proved that $A_k\ge
A_{k+1}\ge A\#B \ge B_{k+1}\ge B_k$ for $k=1,2,\ldots$ (see
\cite{at}), where we say that $P_1\ge P_2$ if $P_1-P_2$ is semidefinite positive.  

The sequences $A_k$ and $B_k$ are related by the simple formulae
$A_k=AB_k^{-1}B=BB_k^{-1}A$ (or equivalently
$B_k=AA_k^{-1}B=BA_k^{-1}A$), which are trivial for $k=0$ and,
assuming them true for $k$, then, the equality
$B_{k+1}^{-1}=(A_k^{-1}+B_k^{-1})/2$ yields
\[
\begin{split}
A_{k+1}=\frac 12(A_k+B_k) & = \frac 12
(AB_k^{-1}B+AA_k^{-1}B)=AB_{k+1}^{-1}B,\\
& = \frac 12 (BB_k^{-1}A+BA_k^{-1}A)=BB_{k+1}^{-1}A,
\end{split}
\]
hence, the formulae are proved by an induction argument.

Using the previous relationships, iteration \eqref{eq:2nd1} can be
uncoupled obtaining the single iterations
\begin{equation}\label{eq:2nd1b1}
    A_0=A \mathrm{\ (or\ }B),\quad A_{k+1}=\frac 12 (A_k+AA_k^{-1}B),
    \quad k=0,1,2,\ldots,
\end{equation}
and
\begin{equation}\label{eq:2nd1b2}
    B_0=A \mathrm{\ (or\ }B),\quad B_{k+1}=2(B_k^{-1}+B^{-1}B_kA^{-1})^{-1},
    \quad k=0,1,2,\ldots
\end{equation}
Iteration \eqref{eq:2nd1b1} has the same computational cost as
\eqref{eq:2nd1}, and seem to be more attractive from a
computational point of view since requires less storage.
However, iterations \eqref{eq:2nd1b1} and \eqref{eq:2nd1b2} are
prone to numerical instability than \eqref{eq:2nd1} as
we will show in Section \ref{sec:test}.

Yet another elegant way to write the averaging iteration is obtained observing that 
\[
B_{k+1}=2A_k(A_k+B_k)^{-1}(A_k+B_k-A_k)=
2A_k-2A_k(A_k+B_k)^{-1}A_k=2A_k-A_kA_{k+1}^{-1}A_k,
\]
which yields the three-terms recurrence
\begin{equation}\label{eq:3terms}
    \left\{
    \begin{array}{l}
    A_0=A,\quad A_1=(A+B)/2,\\
    A_{k+2}=\ds\frac 12 (A_{k+1}+2A_k-A_kA_{k+1}^{-1}A_k),
    \end{array}
    \right.\quad k=0,1,2,\ldots
\end{equation}

Essentially, the same algorithm is obtained applying Newton's
method for computing the sign of a matrix in the following
equality proved by Higham et al. \cite{hmmt05}:
\begin{equation}\label{eq:signBA}
    \mbox{sign}(C)=\twotwo 0{A\#
    B}{(A \#B)^{-1}}0,\quad C:=\twotwo 0B{A^{-1}}0.
\end{equation}
The sign of a matrix $M$ having nonimaginary eigenvalues can be
defined as the limit of the iteration $M_0=M$,
$M_{k+1}=(M_k+M_k^{-1})/2$. Applying the latter iteration to the matrix $C$ of \eqref{eq:signBA} yields a sequence $C_k=\twotwo 0{X_k}{Y_k}0$ and the
coupled iterations
\begin{equation}\label{eq:2nd2}
    \left\{
    \begin{array}{l}
    X_0=B,\quad Y_0=A^{-1},\\
    X_{k+1}=(X_k+Y_k^{-1})/2,\\
    Y_{k+1}=(Y_k+X_k^{-1})/2,
    \end{array}
    \right. \quad k=1,2,\ldots
\end{equation}
where $X_k$ converges to $A\# B$ and $Y_k$ converges to $(A\#
B)^{-1}$.

We prove by induction that the sequences \eqref{eq:2nd1} and
\eqref{eq:2nd2} are such that $X_k=A_k$, $Y_k=B_k^{-1}$, for
$k=1,2,\ldots$ In fact $X_1=(B+A)/2=A_1$,
$Y_1=(A^{-1}+B^{-1})/2=B_1^{-1}$, while $X_{k+1}=
(X_k+Y_k^{-1})/2=(A_k+B_k)/2= A_{k+1}$ and $Y_{k+1}=
(Y_k+X_k^{-1})/2= (A_k^{-1}+B_k^{-1})/2=B_{k+1}^{-1}$.
 
Iteration \eqref{eq:2nd1} based on averaging can be implemented
at the cost per step of three inversion of positive matrices, that is 
$3n^3$ ops, while iteration \eqref{eq:2nd2} based on the
sign function can be implemented at a cost of $2n^3$ ops. Moreover, the scaling
technique for the sign function allows one to accelerate the
convergence. Let $M$ be a matrix such that the sign is well defined, from sign$(M)=$sign$(\gamma M)$ for each $\gamma>0$,
one obtains the scaled sign iteration which is $M_0=\gamma_0 M$,
$M_{k+1} = (\gamma_k M_k +(\gamma_k M_k)^{-1})/2$, where
$\gamma_k$ is a suitable positive number which possibly reduces
the number of steps needed for the required accuracy. A common
choice is the determinantal scaling  $\gamma_k=|\det(M_k)|^{-1/n}$ \cite{byers}, a
quantity that can be computed in an inexpensive way during the
inversion of $M_k$. Another possibility is to use the spectral scaling $\gamma_k=\sqrt{\rho(M_k^{-1})/\rho(M_k)}$ \cite{kl92}, which is interesting in our case since the eigenvalues of $C=\twotwo 0{B}{A^{-1}}0$ are all real and simple (in fact $C^2=\twotwo{BA^{-1}}00{A^{-1}B}$ has only real positive simple eigenvalues) and in this case a theorem of Barraud \cite{barraud,highambook} guarantees the convergence to the exact value of the sign in a number of steps equal to the number of distinct eigenvalues of the matrix.

To get the proper values of the scaling parameters it is enough to observe that $|\det(C_k)|=|\det(X_k)\det(Y_k)|$ and thus for the determinantal scaling $\gamma_k=|\det(X_k)\det(Y_k)|^{-1/(2n)}$, while $\rho(C_k)=\sqrt{\rho(X_kY_k)}$ and thus for the spectral scaling $\gamma_k=\sqrt{ \rho((X_kY_k)^{-1})/\rho(X_kY_k)  }$.

A scaled sign iteration is thus obtained.
\bigskip

{\noindent \bf Algorithm 5.1a (Scaled averaging iteration: sign
based)} Given $A$ and $B$ positive definite matrices. The matrix $A\# B$ is the limit of the matrix iteration
\begin{equation}\label{eq:2nd3}
    \left\{
    \begin{array}{l}
    X_0=B,\quad Y_0=A^{-1},\\
    \gamma_k=\sqrt{ \rho((X_kY_k)^{-1})/\rho(X_kY_k)  } \  (\mbox{or}\ \gamma_k=|\det(X_k)\det(Y_k)|^{-1/(2n)})\\
    X_{k+1}=(\gamma_kX_k+(\gamma_kY_k)^{-1})/2,\\
    Y_{k+1}=(\gamma_kY_k+(\gamma_kX_k)^{-1})/2,
    \end{array}
    \right. \quad k=0,1,2,\ldots
\end{equation}

\bigskip

Using the aforementioned connections between the sign iterates and the
averaging algorithm the scaling can be applied to the latter obtaining
the following three-terms scaled algorithm.

\bigskip

{\noindent \bf Algorithm 5.1b (Scaled averaging iteration:
  three-terms)} Given $A$ and $B$ positive definite matrices. The
matrix $A\# B$ is the limit of the matrix iteration
\begin{equation}\label{eq:3termsscaled}
  \left\{\begin{array}{l}
  \g_k=\left|\displaystyle\frac{\det(A_k)^2}{
  \det(A)\det(B)}\right|^{-1/(2n)},\\[2ex]  A_{0}=A,\quad A_1=
  \displaystyle\frac {\g_1}{2}(\g_0A +B/\g_0),\\[2ex]
  A_{k+2}=\displaystyle\frac{\g_{k+2}}2
  (A_{k+1}+2A_{k}/\g_{k+1}-A_{k}A_{k+1}^{-1}A_{k}),
  \end{array}\right. \quad k=0,1,2\ldots,
\end{equation}

The same sequence is obtained considering the Palindromic Cyclic
Reduction (PCR)
\begin{equation}\label{eq:PCR}
    \left\{\begin{array}{l}
    P_0=\frac 14 (A-B),\quad Q_0=\frac 12 (A+B),\\
    P_{k+1}=-P_kQ_k^{-1}P_k,\\
    Q_{k+1}=Q_k-2P_kQ_k^{-1}P_k,\\
    \end{array}
    \right. \quad k=0,1,2,\ldots
\end{equation}
whose limits are $\lim_{k} Q_k=A\# B$ and $\lim_{k} P_k=0$. This convergence result is rooted on the fact the matrix Laurent polynomial
\begin{equation}\nonumber
\mathcal L(z)=\frac 14(A^{-1}-B^{-1})z^{-1}+\frac
12(A^{-1}+B^{-1})+\frac 14(A^{-1}-B^{-1}) z,
\end{equation}
is invertible in an annulus containing the unit circle and the sequence $Q_k$ of the PCR converges to the central coefficient of its inverse, namely $A\#B$ \cite{im:palin}.

Since the PCR verifies the same three-terms recurrence \eqref{eq:3termsscaled} as
the averaging iteration \cite{im:pcr}, one obtains that $Q_k=A_{k+1}$ and thus $P_k=(A_k-B_k)/4$.

The connection with PCR is useful because allows one to describe more precisely the quadratic convergence of the averaging technique, as stated by the following theorem of Iannazzo and Meini \cite{im:pcr}.
\begin{theorem}\label{thm:avconv}
Let $A$ and $B$ be positive definite matrices, then the PCR sequence $Q_k$ of \eqref{eq:PCR} (and thus the sequence $A_k$ obtained by the averaging iteration \eqref{eq:2nd1}) converges to $A\#B$ and
$\|Q_k-A\#B\|=O(\xi^{2^k})$, where $\xi$ is any real number such that $\rho^2<\xi<1$ with, $\rho=\sigma/(1+\sqrt{1-\sigma^2})$, where $\sigma=\max_{\la\in\sigma(A^{-1}B)}\{|(\la-1)/(\la+1)|\}$.
\end{theorem}

\subsection{Pad\'e approximants to $z^{-1/2}$}

We give another interpretation of the sequences obtained 
by the averaging technique in terms of the Pad\'e appoximants of the function $z^{-1/2}$. To this end, we manipulate the sequence $A_k$ of \eqref{eq:2nd1} showing its connection with Newton's method for the matrix square root and with the matrix sign iteration.

Let $S=A^{-1}B$ and consider the (simplified) Newton method for the square root of $S$, namely
\begin{equation}\label{eq:averanewton}
	\wh A_0=I,\quad \wh A_{k+1}=\frac 12(\wh A_k+\wh A_k^{-1}S).
\end{equation}
The sequence $\wh A_k$ converges to $S^{1/2}$ for any $A$ and $B$, since the eigenvalues of $S$ are real and positive \cite[Thm.  6.9]{highambook}. We claim that $\wh A_k=A^{-1}A_k$, where $A_k$ is one of the two sequences obtained by the averaging iteration.
To prove this fact, a simple induction is sufficient, in fact assuming that $A_k=A\wh A_k$, we have
\[
	A\wh A_{k+1}=\frac 12 (A\wh A_k+A\wh A_k^{-1}S)=
	\frac 12(A_k+AA_k^{-1}AA^{-1}B)=A_{k+1},
\]
in virtue of \eqref{eq:2nd1b1}.

It is well known that Newton's method for the square root of the matrix $S$ \eqref{eq:averanewton} is related to the matrix sign iteration
\[
	Z_{k+1}=\frac 12(Z_k+Z_k^{-1}),\quad Z_0=S^{-1/2},
\]
through the equality $Z_k=S^{-1/2}\wh A_k$ \cite{highambook}, and thus we have that \begin{equation}\label{eq:averasign}
	A_k=A\wh A_k=AS^{1/2}Z_k=(A\# B)Z_k.
\end{equation}

The latter relation allows one to relate the averaging iteration to the Pad\'e approximants to the function $t^{-1/2}$ in a neighborhood of $1$. We use the reciprocal Pad\'e iteration functions defined in \cite{gip} as
\[
\varphi_{2m,2n+1}(z)=\frac{Q_{n,m}(1-z^2)}{zP_{n,m}(1-z^2)},
\]
where $P_{n,m}(\xi)/Q_{n,m}(\xi)$ is the $(n,m)$ Pad\'e approximant to $(1-\xi)^{-1/2}$ at the point $0$, that is
\[
\frac{P_{n,m}(\xi)}{Q_{n,m}(\xi)}-(1-\xi)^{-1/2}=O(\xi^{m+n+1}),	
\]
as $\xi$ tends to $0$ and $P_{n,m}$ and $Q_{n,m}$ are polynomials of degree $n$ and $m$, respectively.

We define the principal reciprocal Pad\'e iteration for $m=n+1$ and $m=n$ as $\wt g_r(z):=\wt g_{m+n+1}(z)= \varphi_{2m,2n+1}(z)$, for which we prove the following composition property.
\begin{lemma}\label{lem:ma}
Let $r,s$ be positive integers. If $r$ is even then $\wt g_{rs}(z)=\wt g_r(\wt g_s(z))$, if $r$ is odd then $\wt g_{rs}(z)=\wt g_{r}\left(\frac 1{\wt g_s(z))}\right)$.
\end{lemma}
\begin{proof}
The principal reciprocal Pad\'e iterations are the reciprocal of the well-known principal Pad\'e iterations, namely 
\begin{equation}\label{eq:grt}
	\wt g_k(z)=\frac{1}{g_k(z)}=\frac{(1+z)^k+(1-z)^k}{(1+z)^k-(1-z)^k}
\end{equation}
where the latter equality follows from the explicit expression of $g_k(z)$ given in \cite[Thm. 5.9]{highambook}. Notice that if $r$ is even, then $g_r(1/z)=g_r(z)$, moreover, $g_{rs}(z)=g_r(g_s(z))$ (in fact it is easy to see that the principal Pad\'e iterations are conjugated to the powers through the Cayley transform $\mathcal C(z)=(1-z)/(1+z)$, that is $g_r(z)=\mathcal C(\mathcal C(z)^r)$), and thus
\[
	\wt g_{r}(\wt g_{s}(z))=\frac{1}{g_{r}(\frac{1}{g_s(z)})}
	=\frac{1}{g_{r}(g_s(z))}=\frac 1{g_{rs}(z)}=\wt g_{rs}(z),
\]
while if $r$ is odd, then  $g_r(1/z)=1/g_r(z)$ and we get $\wt g_{rs}(z)=\wt g_{r}\left(\frac 1{\wt g_s(z))}\right)$.
\end{proof}

We are ready to state the main result of the section where we use $\wt g_2(z)=\frac{1+z^2}{2z}$.
\begin{theorem}\label{thm:pade}
Let $P_k(z)/Q_k(z)$ be the $[2^{k-1},2^{k-1}-1]$ Pad\'e approximant at $0$ to the function $(1-z)^{-1/2}$, with $k>0$, then $A_k=A Q_k(I-A^{-1}B)P_k(I-A^{-1}B)^{-1}$.
\end{theorem}
\begin{proof} 
Let $Z_0=(A^{-1}B)^{-1/2}$.  We prove that $Z_k=\wt g_{2^k}(Z_0)=\varphi_{2^k,2^{k}-1}(Z_0)$, this is true for $k=1$, in fact $Z_1=\wt g_2(Z_0)$, while to prove the inductive step we use Lemma \ref{lem:ma} so that $\wt g_{2^{k+1}}(Z_0)=\wt g_{2}(\wt g_{2^{k}}(Z_0))=\wt g_2(Z_k)=Z_{k+1}$.

Equation \eqref{eq:grt} gives $\wt g_{2^k}(z)=\wt g_{2^k}(1/z)$ and then 
$\wt g_{2^k}(Z_0)=\wt g_{2^k}(Z_0^{-1})=\varphi_{2^{k},2^{k}-1}(Z_0^{-1})$.
Thus, in view of equation \eqref{eq:averasign} and recalling that $A\#B=AZ_0^{-1}$, we have
\begin{multline}\nonumber
	A_k=AZ_0^{-1}Z_k=
\\=AZ_0^{-1}Z_0Q_{2^{k-1},2^{k-1}-1}(I-Z_0^{-2})P_{2^{k-1},2^{k-1}-1}(I-Z_0^{-2})^{-1}
\\=AQ_k(I-A^{-1}B)P_k(I-A^{-1}B)^{-1}.
\end{multline}
\end{proof} 

As a byproduct of the previous analysis we get that the Newton method for the scalar square root is related to the Pad\'e approximation of the square root function.
\begin{Cor}
Let $z\in\C\setminus (-\infty,0]$, and let 
\[
	z_{k+1}=\frac 12(z_
k+zz_k^{-1}),\quad z_0=z,
\]
be the Newton iteration for the square root of $z$, then $z_k=\frac{p(z)}{q(z)}$,
where $p(z)/q(z)$ is the $[2^{k-1},2^{k-1}-1]$ Pad\'e approximant at 1 of the square root function $z^{1/2}$. 
\end{Cor}

\begin{Rem}\rm
Ra\"issouli and Leazizi propose in \cite{rl} an algorithm for the matrix geometric mean which is based on a matrix version of the continuous fraction expansion for scalars $a,b>0$,
\[
\sqrt{ab}=\left[\frac{a+b}2;\frac{-(\frac{a-b}2)^2}{a+b}\right]_{k=1}^{\infty}.
\]
The  partial convergent $t_N=\left[\frac{a+b}2;\frac{-(\frac{a-b}2)^2}{a+b}\right]_{k=1}^{N}$ is proved to be
\[
	t_N=\sqrt{ab}\frac
	{(1+\sqrt{ab})^{2N+2}+(1-\sqrt{ab})^{2N+2}}
	{(1+\sqrt{ab})^{2N+2}-(1-\sqrt{ab})^{2N+2}},
\]
thus from the expression for the Pad\'e approximation in \eqref{eq:grt}, and the characterization of the averaging iteration in terms of the Pad\'e approximation we get that $A_k=t_{2^{k-2}-1}$, for $k\ge 2$, where $A_k$ is one of the sequences obtained by the averaging iteration with $A_0=a$ and $B_0=b$.

The same equivalence holds in the matrix case, so we get that the sequence $t_N$ converges linearly to the matrix geometric mean with a cost similar to the averaging iteration which indeed converges quadratically and moreover can be scaled. Thus, the sequence $t_N$ is of little computational interest.
\end{Rem}

\subsection{Algorithms based on the polar decomposition}

Let $A=R_A^* R_A$ and $B=R_B^*R_B$ be the Cholesky factorizations of $A$ and $B$, respectively.  Using these factorization in formula \eqref{eq:fact} we obtain the following representations for the matrix geometric mean
\begin{equation}\label{eq:pdformula}
A\#B=R_A^*\polar(R_AR_B^{-1})R_B,
=R_B^*\polar(R_BR_A^{-1})R_A
=R_A^*\polar(R_BR_A^{-1})^*R_B,
\end{equation}
where we have used the symmetry of the matrix $A\#B$ and the commutativity of the matrix geometric mean function.

We derive from \eqref{eq:pdformula} an algorithm for computing the
matrix geometric mean.

\bigskip

{\noindent \bf Algorithm 5.2 (Polar decomposition)} Given $A$ and
$B$ positive definite matrices with $\mu(A)\le \mu(B)$.
\begin{enumerate}
  \item Compute the Cholesky factorizations
  $A=R_A^*R_A$ and $B=R_B^*R_B$;
  \item Compute the unitary polar factor $U$ of $R_BR_A^{-1}$;
  \item Compute $A\#B=R_B^*UR_A$.
\end{enumerate}
\bigskip

The polar factor of a matrix $M$ can be computed forming its
singular value decomposition, say $M=Q_1^*\Sigma Q_2$; from which we get the polar
factor of $M$ as $Q_1^*Q_2$. This procedure is
suitable for an accurate computation due to the good numerical
property of the SVD algorithm, but it is expensive with respect to 
a method based on matrix iterations.

A more viable way to compute the unitary polar factor of $M$ is to use the
scaled Newton method
\begin{equation}\label{eq:Npoldec}
    Z_{k+1}=\frac{\gamma_k Z_k+(\gamma_k Z_k)^{-*}}2,\quad Z_0=M,
\end{equation}
where $\gamma_k>0$ can be chosen in order to reduce the number of
steps needed for convergence.

A nice property of the scaled Newton method for the unitary polar factor
of a matrix is that the number of steps can be predicted in
advance for a certain machine precision and the algorithm is backward
stable if the inversion is performed in a mixed backward/forward way (see \cite{kz,hn}).
An alternative is to compute the polar decomposition using a scaled Halley iteration as in \cite{hn}.

The better choice for the scaling factor in the Newton's iteration is the optimal scaling $\gamma_k=(\sigma_1(X_k)\sigma_n(X_k))^{1/2}$, where $\sigma_1(X_k)$ and $\sigma_n(X_k)$ are the extreme singular values of $X_k$. In practice cheaper approximations of the optimal scaling are available \cite[Sec. 8.6]{highambook}.

If $\gamma_k=1$ for each $k$, then the sequence $Z_k$ obtained by
iteration \eqref{eq:Npoldec} with $Z_0=R_BR_A^{-1}$ is strictly
related to the sequence obtained by the averaging
technique, in fact $Z_k=R_B^{-*}A_kR_A^{-1}$, where $A_k$ is
defined in \eqref{eq:2nd1b1} with $A_0=B$.
This
equality can be proved by an induction argument in fact
$R_B^{-*}A_0R_A^{-1}=
R_B^{-*}R_B^*R_BR_A^{-1}=Z_0$
and if the equality is true for $k$, then 
\begin{multline}\nonumber
    Z_{k+1}=\frac 12(Z_k+Z_k^{-*})=R_B^{-*}\left(\frac {A_k}2
    +R_B^*\frac {R_BA_k^{-1}R_A^*}2 R_A
    \right)R_A^{-1}\\
    =R_B^{-*}\left(
    \frac{A_k+BA_k^{-1}A}2
    \right)R_A^{-1}=R_B^{-*}\left(
    \frac{A_k+AA_k^{-1}B}2
    \right)R_A^{-1}=R_B^{-*}A_{k+1}R_A^{-1}.
\end{multline}
The equality $R_B^*Z_kR_A=A_k$, and the monotonicity
of $A_k$ proves that the approximated value of $A\# B$ is greater
than or equal to $A\# B$ in the order of $\mathcal P_n$.

\begin{Rem}\label{rem:higham}\rm
Notice that for $A=I$, Algorithm 5.2 reduces to the algorithm 
of Higham \cite[Alg.
6.21]{highambook} for the square root of a positive matrix $B$.
A side-result of the previous discussion is that Higham's algorithm can be 
seen as yet another variant of the Newton method for the matrix 
square root.
\end{Rem}

\subsection{Gaussian quadrature}

A third algorithm is obtained using the integral representation
\eqref{eq:gm-int} obtained by Ando, Li and Mathias \cite{alm} using an Euler integral,
the same representation is obtained by Iannazzo and Meini \cite{im:palin} from the Cauchy
integral formula for the function $z^{-1/2}$.

The change of variable $z=\frac{t+1}2$ yields
\begin{equation}\label{eq:gm-int2}
  A\#B=\frac 2{\pi}\int_{-1}^1
  \frac{((1+z)B^{-1}+(1-z)A^{-1})^{-1}}{\sqrt{1-z^2}}\,dz,
\end{equation}
which is well suited for Gaussian quadrature with respect to the
weight function $\omega(z)=(1-z^2)^{-1/2}$, referred as
Gauss-Chebyshev quadrature since the orthogonal polynomials with
respect to the weight $\omega(z)$ are the Chebyshev polynomials (see
\cite{gcq} for more details). For an integral of the
form
\[
    \int_{-1}^1 \frac{f(z)}{\sqrt{1-z^2}}\,dz,
\]
where $f$ is a suitable function, the formula is
\[
    \Sigma_{n+1}=\frac {\pi}{n+1}\sum_{k=0}^n f(x_k),\quad
x_k=\cos\left(\displaystyle\frac{(2k+1)\pi}{2(n+1)}\right),\quad
k=0,\ldots,n.
\]

Applying the Gauss-Chebyshev quadrature formula to \eqref{eq:gm-int2} we obtain the following approximation of $A\#B$
\begin{equation}\label{eq:gc}
\begin{split}
    T_{N+1}(A,B) & =\frac {2}{N+1}\sum_{k=0}^N
    ((1+x_k)B^{-1}+(1-x_k)A^{-1})^{-1}\\
    & =B\left(\frac {2}{N+1}\sum_{k=0}^N
    ((1+x_k)A+(1-x_k)B)^{-1}
	\right)A.
\end{split}
\end{equation}

\bigskip

{\noindent \bf Algorithm 5.3 (Gauss-Chebyshev quadrature)} Given $A$ and
$B$ positive definite matrices. Choose $N$ and set
\[
	A\#B \approx T_N(A,B).
\]
where $T_N(A,B)$ is defined in \eqref{eq:gc}
\bigskip

The computation cost is the inversion of a positive matrix, that is $n^3$ ops, for each node of the quadrature and two matrix multiplication at the end. The number of nodes required to get a fixed accuracy depends on the regularity of the function $\psi(z)=((1+z)A+(1-z)B)^{-1}$. The function $\psi(z)$ is rational and thus analytic in the complex plane except the values of $z$ such that  $\psi(z)$ is singular, which are the reciprocal of the nonzero eigenvalues of the matrix $(B-A)(B+A)^{-1}= (A^{-1}B-I)(A^{-1}B+I)^{-1}$. 

We claim that all the poles are real and lie outside the interval $[-1,1]$, which is equivalent to require that the eigenvalues of $(A^{-1}B-I)(A^{-1}B+I)^{-1}$ lie in the interval $(-1,1)$. Define $\mathcal C(z)=(z-1)/(z+1)$, then the image under $\mathcal C(z)$ of the positive real numbers is the interval $(-1,1)$, then the eigenvalues of $\mathcal C(A^{-1}B)$ lie in the interval $(-1,1)$ since $A^{-1}B$ has positive eigenvalues $\la_1,\ldots,\la_n$ and the eigenvalues of $\mathcal C(A^{-1}B)$ are $\mathcal C(\la_1),\ldots,\mathcal C(\la_n)$ (compare \cite[Thm. 1.13]{highambook}).

Standard results on the convergence of the Gauss-Chebyshev quadrature (see \cite[Thm. 3]{gcq}) imply 
that the sequence $T_N(A,B)$ converges to $A\#B$ linearly, in particular for each $ \rho^2<\xi<1$, it holds that $\|T_N(A,B)-A\#B\|=O(\xi^{N})$, where $1/\rho$ is the sum of the semiaxes of an ellipse with foci in $1$ and $-1$ and whose internal part is fully contained in the region of analiticity of $\psi(z)$.

Since the poles of $\psi(z)$ are real and lie outside the interval $(-1,1)$, then the largest ellipse is obtained for $\rho=1/(\frac 1{\sigma}+\sqrt{\frac{1}{\sigma^2}-1})=\sigma/(1+\sqrt{1-\sigma^2})$, where $\sigma=\max\{\mathcal |C(\la_i)|\}$ (notice that $1/\sigma$ is the pole of $\psi(z)$ nearest to $[-1,1]$).

If $m$ and $M$ are the smallest and largest, respectively, eigenvalues of $A^{-1}B$, then the convergence of $T_N(A,B)$ is slow if $m$ is small or $M$ is large. 
By a suitable scaling of $A$, it is possible to have $mM=1$, which gives a faster convergence, however, when $M/m$ tends to infinity the parameter of linear convergence tends to 1, in this case a simple analysis shows that $\rho=1+O\left(\sqrt{\frac Mm}\right)$ and thus the parameter of linear convergence of $T_N(A,B)$ depends linearly on $M/m$. In Section \ref{sec:minimax} we give another quadrature formula whose dependence on $M/m$ is just logarithmic.

A comparison of the parameters of linear convergence for the Gauss-Chebyshev formula $T_N(A,B)$ and the parameters of quadratic convergence for the averaging iteration in Theorem \ref{thm:avconv} reveals that they are essentially the same. This is not a mere
coincidence, in view of the following result.

\begin{thm}\label{thm:GU}
Let $T_k$ be the quadrature formula of \eqref{eq:gc} and $B_k$ be
the sequence obtained by the averaging technique \eqref{eq:2nd1}
then $B_k=T_{2^{k-1}}$, for $k=1,2,\ldots$
\end{thm}
\begin{proof}
Let $S=A^{-1}B$, assume that $S-I$ is invertible, then
\[ \begin{split} B^{-1}T_N&
=\frac
2N\sum_{k=0}^{N-1}((1+x_k)I+(1-x_k)S)^{-1}\\
&=\frac 2N \sum_{k=0}^{N-1}(I+S+x_k(I-S))^{-1},\\
&=\frac 2N(I-S)^{-1} \sum_{k=0}^{N-1}(\mathcal K(S)+x_kI)^{-1},
\end{split}
\]
where $\mathcal K(z)=(1+z)/(1-z)$. Let $\mathcal T_N(x)$ be the
$N$th Chebyshev polynomial, then $\sum_{k=0}^{N-1}
(x_k+t)^{-1}=\mathcal T'_N(t)/\mathcal T_N(t)$, thus
\[ \begin{split} B^{-1}T_N &=\frac 2N(I-S)^{-1}
\mathcal T'_N(\mathcal K(S)) \mathcal T_N(\mathcal K(S))^{-1}.
\end{split}
\]
To conclude the proof, since $T_1=B_1$ by direct inspection, it is enough to prove by induction that
$T_{2^k}=2(T_{2^{k-1}}^{-1}+A^{-1}T_{2^{k-1}}B^{-1})^{-1}$ which
is equivalent to prove that $2(B^{-1}T_{2^k})^{-1}=
(B^{-1}T_{2^{k-1}})^{-1} +SB^{-1}T_{2^{k-1}}$ ($T_{2^{k}}$ is invertible since the zeros of $\mathcal T'_{2^k}(t)$ lie in $(-1,1)$). Observe that
$\mathcal T'_{2^k}(t)=4\mathcal T'_{2^{k-1}}\mathcal T_{2^{k-1}}$,
and thus 
\[
2(B^{-1}T_{2^k})^{-1}=2^{k-2}(I-S)\mathcal {T'}_{2^{k-1}}(\mathcal K(S))^{-1}\mathcal T_{2^{k-1}}(\mathcal K(S))^{-1}\mathcal T_{2^k}(\mathcal K(S)).
\]
On the other hand
\begin{multline}\nonumber
(B^{-1}T_{2^{k-1}})^{-1}+ SB^{-1}T_{2^{k-1}}\\ =  
2^{k-2}(I-S)\mathcal {T'}_{2^{k-1}}(\mathcal K(S))^{-1}
\mathcal T_{2^{k-1}}(\mathcal K(S))
+ 
2^{2-k}S(I-S)^{-1}
\mathcal {T'}_{2^{k-1}}(\mathcal K(S))
\mathcal T_{2^{k-1}}(\mathcal K(S))^{-1}.
\end{multline}
After some manipulations, to conclude it is enough to prove that
\begin{equation}\label{eq:ChPf0}
  \mathcal T_{2^k}(\mathcal K(S))=
    \mathcal T_{2^{k-1}}(\mathcal K(S))^2+2^{4-2k}S(I-S)^{-2}
    \mathcal T'_{2^{k-1}}(\mathcal K(S))^2 ,
\end{equation}
this property is a special case of a more general identity
involving Chebyshev polynomials, in fact for each $k$, it holds
that
\begin{equation}\label{eq:ChPf1}\mathcal T_k(\mathcal
K(z))^2=\ds \frac {4z}{k^2(z-1)^2}{{\mathcal T}'_k(\mathcal
K(z))}^2+1\end{equation} for a complex variable $z$. 
By the change of variable $x=\mathcal K(z)$ formula
\eqref{eq:ChPf1} is equivalent to $\mathcal T_k(x)^2=\ds \frac
{x^2-1}{k^2}{{\mathcal T}'_k(x)}^2+1$ which can be proved directly
using $\mathcal T_k(x)=\cos(k\arccos x)$. The equation
\eqref{eq:ChPf1} implies \eqref{eq:ChPf0} in fact $\mathcal
T_{2^k}=2\mathcal T_{2^{k-1}}^2-1$.

If $S-I$ is singular, then $\beta A^{-1}B-I$ is invertible in a neighborhood of $\beta=1$ except $\beta=1$, thus $T_{2^{k-1}}(A,\beta B)=B_k(A,\beta B)$, which gives the desired equality as $\beta$ tends to 1.
\end{proof}


\section{Algorithms based on the rational minimax
approximation of $z^{-1/2}$}\label{sec:minimax}

In Section \ref{sec:pade} we have found that many algorithms for computing the matrix geometric mean are variations of the one obtained by using certain Pad\'e approximations of $z^{-1/2}$ in the formula $A\#B=A(B^{-1}A)^{-1/2}$. To get something really different, one should change the rational approximation. The natural direction is towards the (relative) rational minimax approximation.

Let $R_{k-1,k}$ be the set of rational functions whose numerator and denominator have degree $k-1$ and $k$, respectively. The function $\wt r_{k,\gamma}(z)$ is said to be the rational relative minimax approximation to $z^{-1/2}$ in the interval $[1,\gamma]$ if it minimizes over $R_{k-1,k}$
the quantity
\[
	\max_{z\in [1,\gamma]}\left|\frac{r(z)-z^{-1/2}}{z^{-1/2}}\right|.
\]
An explicit expression for $\wt r_{k,\gamma}(z)$, in terms of elliptic function is known since the work of Zolotarev in 1877 (see \cite{qcd}).

The same approximation is obtained by Hale, Higham and Trefethen \cite{triplenick} by a trapezoidal quadrature following a clever sequence of substitutions applied to the Cauchy integral formula for $A^{-1/2}$, namely, 
\[
	A^{-1/2}=\frac 1{2\pi \iunit}\oint_{\Gamma} z^{-1/2}(zI-A)^{-1}dz.
\]
Since $A\# B=A(A^{-1}B)^{1/2}=B(A^{-1}B)^{-1/2}$, using the results of \cite{triplenick}, we get the following approximation (obtained by a quadrature formula on $N$ nodes on a suitable integral representation of $A\# B$)
\begin{equation}\label{eq:tref}
	S_N(A,B)=  B\left(\frac {-2K'\sqrt{m}}{\pi N}\sum_{j=1}^N (\omega(t_j)^2A-B)^{-1}\mbox{cn}(t_j)\mbox{dn}(t_j)\right)A
\end{equation}
which is proved to coincide with $A\,r_{N,\gamma}({B^{-1}A})$ for $\gamma=M/m$, where $M$ and $m$ are the largest and the smallest eigenvalues of $A^{-1}B$, respectively.

The notation of \eqref{eq:tref} has the following meaning:
\[
	t_j=\bigl(j-1/2\bigr)\frac {K'}N \iunit,\quad 1\le j\le N,
\]
$w(t_j)=\sqrt{m}\,$sn$(t_j|\gamma)$, where sn$(t_j|\gamma)$, cn$(t_j|\gamma)$ and dn$(t_j|\gamma)$ are the Jacobi elliptic functions, while $K'$ is the complete elliptic integral of the second kind associated with $\sqrt{\gamma}$ (see \cite{as} for an introduction to elliptic functions and integrals).

The convergence of $S_N(A,B)$ to $A\#B$ can be deduced from Theorem 4.1 of \cite{triplenick}. In particular,
\[
	\|A\# B- S_N(A,B)\|=O(e^{-2\pi^2N/(\log(M/m)+3)}).
\]
Thus, the convergence of the sequence $S_N(A,B)$ to $A\#B$ is dominated by a sequence whose convergence is linear with a rate which tends to $1$ as $M/m$ tends to $\infty$, but whose dependence on $M/m$ is just logarithmic. On the contrary, the rate of linear convergence of the Gauss-Chebyshev sequence $T_N(A,B)$ of \eqref{eq:gc} depends linearly on $M/m$, and thus we expect that the formula $S_N(A,B)$ requires less nodes than $T_N(A,B)$ to get the same accuracy on the approximation of $A\#B$ at least for large values $M/m$. In practice, the approximation obtained from $T_N(A,B)$ is always better than $S_N(A,B)$ as suggested by our numerical tests of Section \ref{sec:test}.

We describe the synthetic algorithm.

\bigskip

{\noindent \bf Algorithm 6.1 (Rational minimax)} Given $A$ and
$B$ positive definite matrices. Choose $N$ and set
\[
	A\# B \approx S_N,
\]
where $S_N$ is defined in \eqref{eq:tref}.


\section{Applications}\label{sec:appl}

We review some of the applications in which the geometric mean
of two matrices is required, they range from electrical network analysis \cite{amt} to medical imaging \cite{afpa}, from norm on fractional Sobolev spaces \cite{al} to image deblurring \cite{dbe}, to the computation of the geometric mean of several matrices \cite{alm,bmp}, with indirect applications to radar \cite{barbaresco} and 
elasticity \cite{moa2}.

\subsection{Electrical networks}

Fundamental elements of a circuit are the resistances which can be modeled by positive real numbers. It is a customary high school argument that two consecutive resistances $r_1$ and $r_2$ in the same line can be modeled by a unique joint resistance whose value is the sum $r_1+r_2$, while if the two resistances lie in two parallel lines their joint resistance is the ``parallel sum'' $(r_1^{-1}+r_2^{-1})^{-1}$.

More sophisticated devices based on resistances are $n$-port networks, which are ``objects'' with $2n$ ports at which current and voltage can be measured, without knowing what happens inside. The usual way to model $n$-port networks is through positive definite matrices. In this way two consecutive $n$-ports $A$ and $B$ can be modeled as the joint $n$-port $A+B$, while two parallel $n$-ports give the joint $n$-port $(A^{-1}+B^{-1})^{-1}$. 

Complicated circuits, made of several $n$-ports can be reduced to a joint $n$-port using these sums and parallel sums. Consider the circuit in Figure \ref{fig:1}: it is an infinite network (which models a large finite network). 

\begin{figure}
\begin{center}
\includegraphics[width=6cm]{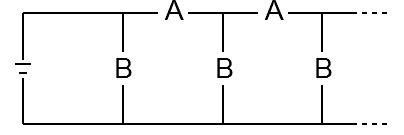}
\end{center}
\caption{An electric circuit whose joint resistance is related to the matrix geometric mean}\label{fig:1}
\end{figure}

Let $Z_k$ be the joint resistance of the subcircuit obtained selecting the first $k$ loops, then it can be shown that $Z_1=B$ and
\[
	Z_{k+1}=(B^{-1}+(A+Z_k)^{-1})^{-1},
\]
and the sequence has limit $\lim_{k\to \infty}Z_k=\frac 12(-A+(A\#(A+4B)))$. This limit is the joint resistance of the infinite circuit. For further details see \cite{amt82}, from which the example is taken.

It is worth pointing out that the definition of geometric mean of two matrices first appeared in connection with these kind of applications \cite{pw}.

\subsection{Diffusion tensor imaging}

The technique of Nuclear Magnetic Resonance (NMR) in medicine produces
images of some internal parts of the body which are used by medics
to give a diagnose of important pathologies or to decide how to perform
a surgery.

One of the quantities measured by the NMR is the diffusion tensor
which is a $3\times 3$ positive matrix describing the diffusion of the
water in tissues like the white matter of the brain or the
prostate. The technique is called Diffusion Tensor Imaging (DTI).

The diffusion tensor is measured for any of the points of an ideal
grid into the tissue, thus one has a certain number of positive matrices
indexed by their positions. 

A problem in DTI is the ``interpolation'' of tensors, that is, given
two tensors, find one or more tensors in the line joining them, the
more adherent to the real data as possible. This is useful for instance
to increase the resolution of an image or to reconstruct some
corrupted parts.

Many models have been given for the interpolation of tensors in DTI,
the most obvious of which is the linear interpolation, where $k$
points between $A$ and $B$ are $P_j=\frac j{k+1}A+\frac{k+1-j}{k+1}B$,
for $j=1,\ldots,k$. The linear interpolation finds point equally
spaced on the line joining $A$ and $B$ in the space of the matrices,
that is, uses the Euclidean geometry of $\C^{n\times n}$. 

Some more adequate models use Riemannian geometries. Using the
geometry given in Section \ref{sec:intro} we get the interpolation points
\[
   P_k^{a}=A\#_{j/(k+1)}B=A(A^{-1}B)^{j/(k+1)},\quad j=1,\ldots k.
\]
Using the {\em log Euclidean} geometry defined in \cite{afpa}, we get the interpolation points
\[
   P_k^{b}=\exp\left(\frac j{k+1}\log(A)+\frac{k+1-j}{k+1}\log(B)\right),\quad
   j=1,\ldots k.
\]
The log Euclidean geometry has been introduced as an approximation to the Riemannian geometry where quantities are easier to be computed. However, in the interpolation problem described here, using the Cholesky-Schur algorithm of Section \ref{sec:schur} to compute $P_k^a$ (reusing the Schur factorization of $R_A^{-*}BR_A^{-1}$ for each $j$) is much less expensive than the computation of $P_k^b$ using the customary algorithms for the logarithm and the exponential of a matrix.

\subsection{Computing means of more than two matrices}

The generalization of the geometric mean to more than two positive matrices is usually identified with their Karcher mean in the geometry given in Section \ref{sec:intro} (see \cite{bi} for a precise definition). 

The Karcher mean of $A_1,\ldots,A_m$ can been obtained as the limit of the sequence $S_k=S_{k-1}\#_{1/k} A_{(k\mathrm{\ mod\ }m)+1}$, with $S_1=A_1$ as proved by Holbrook \cite{holbrook}. The resulting sequence is very slow and cannot be used to design an efficient algorithm for the computation of the Karcher mean, however it may be useful to construct an initial value for some other iterative methods like the Richardson-like iteration of Bini and Iannazzo \cite{bi}.

Other geometric-like means of more than two matrices are based on recursive definitions like the mean proposed by Ando, Li and Mathias \cite{alm}, which for three matrices $A_0,B_0$ and $C_0$ is defined as the common limit of the sequences 
\[
	A_{k+1}=B_k\# C_k,\quad B_{k+1}=C_k\# A_k,\quad C_{k+1}=A_k\# B_k,\quad k=0,1,2,\ldots
\]
These sequences converge linearly to their limit. Another similar definition which gives cubic convergence (to a different limit) has been proposed in \cite{bmp,nakamura}., who propose the iteration
\[
	A_{k+1}=A_k\#_{2/3}(B_k\# C_k),\quad
	B_{k+1}=B_k\#_{2/3}(C_k\# A_k),\quad
	C_{k+1}=C_k\#_{2/3}(A_k\# B_k).
\]
As one can see, the efficient computation of $A\#_t B$ is a basic step to implement these kind of iterations.

\subsection{Image deblurring}

A classical problem in image processing is the image deblurring which consists in finding a clear image from a blurred one. In the classical models, the true and the blurred images are vectors and the blurring operator $A$ is linear, thus the problem is reduced to the linear system $Af=g$ which in practice is very large and ill-conditioned. A computationally easy case is the one in which $A$ is a band Toeplitz matrix, which corresponds to the so-called shift-invariant blurred operators.

Even if $A$ is not shift-invariant, it can be possible, in certain cases, that a change of coordinate $M$ makes it shift-invariant, i.e. $M^TAM$ is band Toeplitz. If such a $M$ exists and is known, then the linear system $Af=g$ has the same nice computational properties as a band Toeplitz system.

When $A$ and $T$ are positive definite, the matrix $M=A^{-1}\#T$ is an explicit change of coordinates. For further details see \cite{dbe}.

\subsection{Discrete interpolation norm}

The material of this section is taken from \cite{al} to which we address the reader for a full detailed description.

Let $\Omega\subseteq\R^n$ be open, bounded and with smooth boundary, and let $H_0^1(\Omega)$ be the Sobolev space of differentiable functions on $L^2(\Omega)$ with zero trace, while $H_0^0(\Omega)$ be the set of functions on $L^2(\Omega)$ with zero trace.
 
Let $\{\varphi_1,\ldots,\varphi_n\}$ be a set of linearly independent piecewise linear polynomials on a suitable subdivision of $\Omega$ (arising, for instance, from a finite elements method), then the span in $H_0^1$ (resp. $H_0^0$) of $\{\varphi_i\}_{i=1,\ldots,n}$, is an Hilbert subspace $X_h$ (resp. $Y_h$).

Define the matrices $L_0$ and $L_1$ such that
\[
	(L_0)_{ij}=\langle\varphi_i,\varphi_j \rangle_{L^2(\Omega)},\qquad
	(L_1)_{ij}=\langle \nabla \varphi_i,\nabla \varphi_j \rangle_{L^2(\Omega)}.
\]
The matrices $L_0$ and $L_1$ are positive definite since they are Grammians with respect to a scalar product, in particular $L_0$ is a discrete identity and $L_1$ is a discrete Dirichlet Laplacian. A norm for the interpolation space $[X_h,Y_h]_\theta$ is given by the energy norm of the matrix
\[
	L_0(L_0^{-1}L_0)^{1-\theta}=L_0\#_{1-\theta} L_1.
\]
The most interesting case is $\theta=1/2$, where the norm is given by the geometric mean of $L_0$ and $L_1$.

A similar construction can be used to generate norm of interpolation spaces between finite dimensional subspaces of generic Sobolev spaces with applications to preconditioners of the Stenkov--Poincar\'e operator or boundary preconditioners for the biharmonic operator.


\section{Numerical Experiments}\label{sec:test} 

We present some numerical tests to illustrate the behavior in finite precision arithmetic of the algorithms presented in the paper. The tests have been performed using GNU Octave 3.2.3 on a 2008 Laptop. The scripts of the tests are available at the author's personal web page, so that any numerical experiment can be easily replicated by the reader. The implementations are not efficient, but they are made just to test the behavior of the algorithms. Regarding Algorithm 6.1, based on the rational minimax approximation, we have used the code of \cite{triplenick}.
The implementation of the best algorithms for the matrix geometric mean can be found  also in the Matrix Means Toolbox \cite{mmtoolbox}.

As a measure of accuracy we consider  the relative error $\|\wt G-G\|/\|G\|$, where $\wt G$ is the computed value of the geometric mean, while $G$ is the {\em exact} (up to the roundoff) solution obtained by a direct formula.

\begin{Test}\label{test:1}\rm
We want to compare the behavior of the algorithms showing how the convergence of the iterative algorithms and quadrature formulae depends on the quotient $M/m$, where $M$ and $m$ are the largest and the smallest eigenvalues of $A^{-1}B$. 

We consider, for $x>1/2$, the matrices
\[
	A=\twotwo 2112,\quad B=\twotwo x112,
\]
whose corresponding geometric mean and $A^{-1}B$ are
\[
	A\#B=\twotwo{\frac 12(1+\sqrt{6x-3})}112,\quad
	A^{-1}B=\twotwo{\frac 13(2x-1)}0{\frac 13(2-x)}1.
\]
For $x\ge 2$, we have $M/m=\frac 13(2x-1)$.

For $x=10$ and $x=1000$, we compute an approximation of $A\# B$  in double precision using the different algorithms and monitor the relative error at each step for the iterations and for an increasing number of nodes for quadrature rules. The results are drawn in Figure \ref{fig:test1}, where the algorithm considered are: the Averaging algorithm (AV), namely iteration \eqref{eq:2nd1}; the Averaging iteration with spectral scaling (AVs), namely Algorithm 5.1a; the polar decomposition algorithm (PD), namely Algorithm 5.2, where the polar factor is computed by Newton's method with spectral scaling; the rational minimax approximation algorithm (MM), namely Algorithm 6.1; and the Gauss-Chebyshev quadrature (GC), namely Algorithm 5.3.

As one can see, the convergence of iterations and quadrature formulae is strictly related to the quotient $M/m$ as the analysis suggests. Both quadrature rules show linear convergence, but the one based on rational minimax is much more effective. Regarding the scalings of the averaging/sign iteration, the fast convergence of the spectral scaling fits the fact that this case is made of $2\times 2$ matrices and hence two steps are sufficient for the convergence. Nevertheless, the spectral scaling has given better convergence in all of our experiments with respect to the determinantal scaling.

\begin{figure}
\begin{center}
\includegraphics[width=13cm,height=6cm]{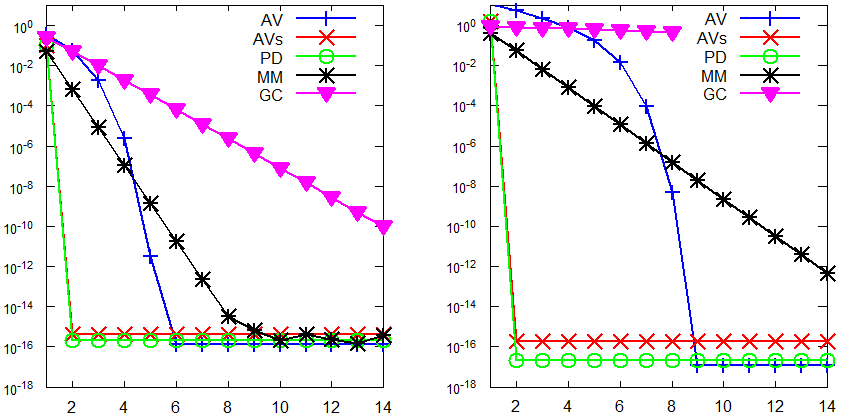}
\end{center}
\label{fig:test1}
\caption{Comparison of the accuracy obtained at various steps by the Averaging iterations (AV), by the Averaging iteration with spectral scaling (AVs), by the polar decomposition algorithm (PD) and the accuracy for various number of nodes of the rational minimax approximation algorithm (MM) and of the Gauss-Chebyshev quadrature (GC) on Test \ref{test:1} for $x=10$ (left) and $x=1000$ (right).}
\end{figure}

\end{Test}

\begin{Test}\label{test:2.5}\rm
Now we want to test the algorithms in some tough problems, we consider the identity matrix $I$ and a diagonal matrix whose diagonal elements are equally spaced between $1$ and $t$, for $t>1$. We test the algorithms for the couple $A=MM^*$, $B=MDM^*$, whose matrix mean is $MD^{1/2}M^*$, and where $M$ the Hilbert matrix which is a classical example of a very ill-conditioned matrix. The exact solution can be computed accurately since $A\#B=MD^{1/2}M^*$, and thus the relative error gives a genuine measure of the accuracy of the algorithms. 

We use the same algorithms as Test \ref{test:1} removing the one based on Gauss-Chebyshev quadrature, since it is much less efficient, and adding the Cholesky-Schur method (Algorithm 4.1) whose great stability guarantees the best forward error. In Figure \ref{fig:test2.5} there is a comparison of methods for $5\times 5$ matrices and for the values $t=10^2$ and $t=10^4$. In the case $t=10^2$ the relative condition number, as defined in Section \ref{sec:cond} is $1.5\cdot 10^6$ and the lowest error (about $10^{-9}$) is obtained by the Cholesky-Schur method, a similar accuracy with a lower computational cost is obtained by the polar decomposition, while the other algorithms seem to have more difficulties. The results are similar for $t=10^4$, the only difference is that now the convergence is slower and the conditioning is greater and thus the numerical results are poorer.

What we have experimented in most of the tests is that, besides the Cholesky-Schur method, the polar decomposition method where the polar factor is computed by Newton's method with  spectral scaling performs better than the other methods.

\begin{figure}
\begin{center}
\includegraphics[width=13cm,height=6cm]{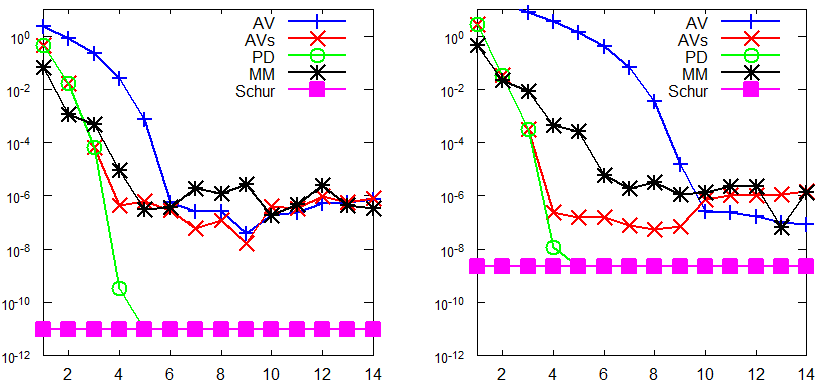}
\end{center}
\label{fig:test2.5}
\caption{Comparison of the accuracy obtained at various steps by the Averaging algorithm (AV), by the Averaging iteration with spectral scaling (AVs), by polar decomposition algorithm (PD) and the accuracy for various number of nodes of the rational minimax approximation algorithm (MM) and the Cholesky-Schur method (Schur) on Test \ref{test:2.5} with $t=10^2$ (left) and $t=10^4$ (right)}
\end{figure}

\end{Test}

\begin{Test}\label{test:3}\rm
We want to address the stability issues related to the iterations presented in Section \ref{sec:avera}. In fact, proving that the sequence (orbit) $\{X_k\}$ obtained by a matrix iteration $X_{k+1}=G(X_k)$, with a given $X_0$, converges in exact arithmetic is not sufficient to guarantees the numerical convergence. This fact has been observed for the (simplified) Newton method for matrix roots and has been first explained by Higham \cite{higham86}. The reason of the numerical failure is that the limit of the iteration is not a stable fixed point, in the sense that the derivative of $G$ at the fixed point has spectral radius larger than one and thus there are points $Y$ in any neighborhood of $G$ such that $G(Y)$ gets far from the fixed point. In finite arithmetic, rounding errors may cause a deviation from $X_k$ to a nearby point $\wt X_k$ whose orbit diverge.

We compute the derivative at $A\#B$ of the iterations defining the averaging iteration \eqref{eq:2nd1} and its uncoupled variant \eqref{eq:2nd1b1} showing that the first has spectral radius less than one (and so it is stable) for any $A$ and $B$, while the second has spectral radius greater than one for certain $A$ and $B$.

Define $G(X,Y)=[(X+Y)/2\ \ \ 2X(X+Y)^{-1}Y]$ such that the averaging iteration is $[A_{k+1}\ B_{k+1}]=G(A_k,B_k)$ then 
\begin{multline}\nonumber
	dG_{[X\ Y]}[H,K]\\=\!\left[\frac 12 (H\!+\!K)\ \ 2H(X\!+\!Y)^{-1}Y\!-\!2X(X\!+\!Y)^{-1}(H\!+\!K)(X\!+\!Y)^{-1}Y\!+\!2X(X\!+\!Y)^{-1}K\right]\!,
\end{multline}
from which we get in the vec basis $dG_{[A\#B\ A\#B]}=\frac 12\twotwo 1111$ whose spectral radius is one. This fact let us expect that a small perturbation on the iterates $A_k$ and $B_k$ of \eqref{eq:2nd1} near to the geometric mean is not amplified in the successive iterates.

On the other hand, define $F(X)=\frac 12(X+AX^{-1}B)$, then $dF_X[H]=\frac 12(H-AX^{-1}HX^{-1}B)$, in the vec basis we have
\[
	dF_{A\#B}=\frac 12(I-B(A\#B)^{-1}\otimes A(A\#B)^{-1})=\frac 12(I-Z\otimes Z^{-1}),
\]
where $Z=(BA^{-1})^{1/2}$. The eigenvalues of $dF_{A\#B}$ are of the form $\frac 12(1-\la_i/\la_j)$ where $\la_i,\la_j$ are any two eigenvalues of $Z$. Since the eigenvalues of $Z$ are real we get that $\rho(dF_{A\#B})\le 1$ if $\frac 12|1-\la_M/\la_m|\le 1$, that  is $\la_M/\la_n\le 3$, where $\la_M$ and $\la_m$ are the largest and the smallest eigenvalues of $Z$, respectively. Thus, we expect numerical instability for matrices $A$ and $B$ such that the quotient $\la_M/\la_n$ is greater than 3.

We consider the matrices of Test \ref{test:2.5} with $n=5$ and where the diagonal elements of $D$ are logarithmically spaced between $1$ and $10^{-t}$, for $t=0.5$ and $t=1.5$. In the former case we get $\rho(dF_{A\# B})\approx 1.8\le 3$, in the latter $\rho(dF_{A\#B})\approx 5.6>3$, and in fact in the first case the uncoupled averaging iteration \eqref{eq:2nd1b1} performs stably, while in the second case it reveals instability. In both cases the standard averaging iteration is stable. The results are drawn in Figure \ref{fig:test3}.

\begin{figure}
\begin{center}
\includegraphics[width=13cm,height=6cm]{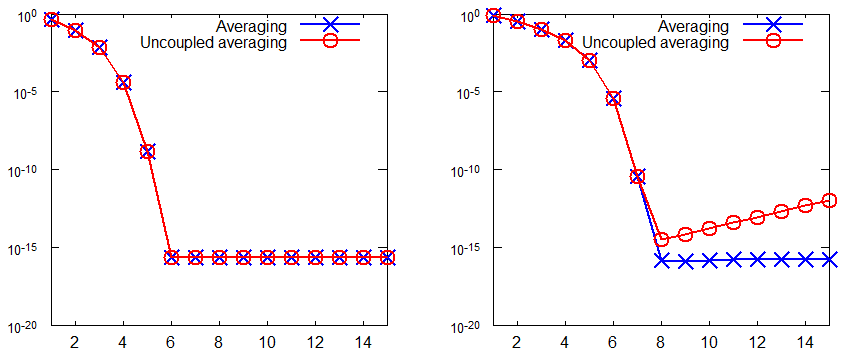}
\end{center}
\label{fig:test3}
\caption{Iteration step vs. relative error for the averaging iteration and its uncoupled version for $t=0.5$ (left) and $t=1.5$ (right) as in Test \ref{test:3}. In the second case the derivative of $G(X)$ has spectral radius greater than one and the uncoupled iteration shows numerical instability.}
\end{figure}

\end{Test}


\section{Conclusions}\label{sec:conclusions}

We have studied the computational issues related to the matrix geometric mean of two positive definite matrices $A$ and $B$, from the conditioning to the classification of the numerical algorithms for computing $A\# B$. We have analyzed many algorithms, most of which are new, or have not yet been considered in the literature. The algorithms are either based on the Schur decomposition or are iterations or quadrature formulae converging to the geometric mean. A very nice fact is that all iterations and quadrature formulae we were able to found were related to the two important rational approximation of $z^{-1/2}$, namely, the Pad\'e approximation and the rational relative minimax approximation.

We have observed that the Pad\'e approximation requires a much high degree than the rational relative minimax approximation to get the same accuracy. On the other hand, the advantage of the Pad\'e approximation is that there exists a recurrence relation between the $[2^k,2^k-1]$ Pad\'e approximants to $z^{-1/2}$ and this recurrence leads to a quadratically convergent algorithm which outperforms the one based on rational minimax approximation. The quadratically convergent iterations can be scaled to get very efficient algorithms, as the one based on the polar decomposition of a suitable matrix. 

Our preferred algorithms for computing the matrix geometric mean are the one based on the Schur decomposition, namely the Cholesky-Schur algorithm, and the ones based on the scaled averaging and scaled polar decomposition, although for large matrices it may be necessary to use a quadrature formula as the rational minimax approximation. A better understanding of the problem $(A\#B)v$ with $A$ and $B$ large and sparse matrices and $v$ a vector is needed and is the topic of a future work. 

We wonder if some kind of recurrence could be found for the rational relative minimax approximation. Moreover, the algorithms based on the Pad\'e approximation benefit considerably by the scaling technique. One might wonder what is the interpretation of the scaling in terms of the approximation and if it is possible to get a ``scaled rational minimax'' approximation in order to accelerate the convergence.

Another issue is related to the equivalence of methods. For this problem we have found the equivalence between a Newton method, a Pad\'e approximation, the Cyclic Reduction and a Gaussian quadrature. We wonder if this intimate connection is true in more general settings. For instance, it would be nice to see the Cyclic Reduction algorithm as a function approximation algorithm.


\section*{Acknowledgments} The author wish to thank George Trapp who kindly sent him some classical papers about the matrix geometric mean and Elena Addis a student who defended a thesis on these topics and who gave the remarkable
quote about the interpretation of the geometric mean as the mid-point of a geodesic:
\begin{quote}
It fills of geometric meaning what of geometric had just the
name.
\end{quote}

\bibliographystyle{abbrv}
\bibliography{matrixmean}

\end{document}